\documentclass[12pt]{amsart}
\usepackage{fullpage}
\usepackage{amsmath, amsthm, amssymb, enumerate, amsfonts, url}
\usepackage[hidelinks]{hyperref}

\def\ZZ{{\mathbb Z}}

\def\CC{{\mathbb C}}

\def\RR{{\mathbb R}}
\def\QQ{{\mathbb Q}}
\def\Qbar{{\overline{\mathbb Q}}}

\def\QQ{{\mathbb Q}}
\def\Qtr{{{\mathbb Q^{\operatorname{tr}}}}}

\def\Qab{{{\mathbb Q^{\operatorname{ab}}}}}
\def\Zab{{{\mathbb Z^{\operatorname{ab}}}}}
\def\Z{{\mathbb Z}}

\def\iffdef{\stackrel{\operatorname{def}}{\iff}}

\def\OO{{\mathcal O}}

\DeclareMathOperator{\disc}{disc}
\DeclareMathOperator{\Gal}{Gal}

\DeclareMathOperator{\JR}{JR}

\DeclareMathOperator{\Assoc}{Assoc}

\DeclareMathOperator{\Tor}{Tor}

\DeclareMathOperator{\ord}{ord}

\DeclareMathOperator{\PrPo}{PrPo}
\DeclareMathOperator{\Level}{Level}

\DeclareMathOperator{\Group}{Group}

\newtheorem{theorem}{Theorem}[section]
\newtheorem*{theorem*}{Main Theorem}

\newtheorem{lemma}[theorem]{Lemma}
\newtheorem{prop}[theorem]{Proposition}
\newtheorem{corollary}[theorem]{Corollary}
\newtheorem{definition}[theorem]{Definition}

\newtheorem{remark}[theorem]{Remark}

\usepackage[dvipsnames]{xcolor}

\author{Caleb Springer}
\address{Center for Communications Research, Princeton, NJ 08540}
\email{c.springer@idaccr.org}
\title{Definability and undecidability via the torsion subgroup of units}
\date{4 September 2026}

\hypersetup{pdftitle=Definability and undecidability via the torsion subgroup of units,pdfauthor={Caleb Springer}}

\begin{document}

\begin{abstract}
In this paper, we prove that $\ZZ$ is first-order definable in the ring of integers $\Zab$ of the maximal abelian extension $\Qab$ of $\QQ$, which implies that the first-order theory of $\Zab$ is undecidable.
More generally, writing $i = \sqrt{-1}$ and $\Qtr$ for the field of all totally real numbers, we prove new definability and undecidability results for rings of integers of subfields of $\Qtr(i)$, focusing especially on fields which contain infinitely many roots of unity.
The key ingredient for these results is that there is a parameter-free positive-existential formula which defines the roots of unity $\mu(\OO_L)$ inside $\OO_L$ for every field $L\subseteq \Qtr(i)$. 
\end{abstract}
\maketitle

\section{Introduction}
\subsection{A motivating example}
Let $\Qab$ denote the maximal abelian extension of $\QQ$.
By the famous Kronecker--Weber theorem, $\Qab$ is precisely the extension of $\QQ$ generated by all roots of unity, i.e., the algebraic numbers $\zeta$ which satisfy $\zeta^n = 1$ for some $n$.
In particular, $\Qab$ is a totally imaginary quadratic extension of a totally real field, meaning it is a subfield of $\Qtr(i)$ where $\Qtr$ is the field of all totally real numbers and $i = \sqrt{-1}$.

The problem of the first-order decidability (in the language of rings) of the ring of integers $\Zab$ of $\Qab$ has been an open problem of interest for several years, as noted in \cite[\S6.3]{Koenigsmann14}, \cite[Ques.~3.1]{Kartas21}, \cite[\S1]{Springer24} and \cite[Rem.~2.13]{MRS24}.
The guiding motivation for this paper is to resolve this question.

\begin{theorem}\label{thm:main_Zab}
	$\ZZ$ is definable in $\Zab$. Thus, the first-order theory of $\Zab$ is undecidable.
\end{theorem}

Because $\ZZ$ has undecidable first-order theory, undecidability follows immediately from definability. 
However, as discussed in the following subsections, there are two more general constructions sitting behind this example, one focusing on definability of $\ZZ$ in rings of integers of algebraic extensions of $\QQ$, and the other focusing on establishing undecidability. 

The core technique of this paper is a unit-group-based method, similar to \cite{MRUV20,MRS24,Springer20,Springer24}. 
One of the main results of our previous paper \cite{Springer24} showed that the ring of integers $\OO_{\Qtr(i)}$ is not definable in $\Qtr(i)$, building on work of J. Robinson \cite{Robinson62} and Fried, Haran and V\"{o}lklein \cite{FHV94}. 
This can also be deduced independently from work of Dittmann and Fehm \cite{DF21}.
Currently, as far as we know, $\Qab$ is amenable to neither these known techniques of establishing non-definability of a ring of integers, nor the techniques of proving definability such as \cite{Shlapentokh18}.
We leave the question of the definability of $\Zab$ in $\Qab$, as well as the first-order decidability of $\Qab$ itself, as open questions for the future.

\subsection{Defining the torsion subgroup of units}

Given any set $S$ of algebraic numbers, let $\mu(S)$ be the set of roots of unity in $S$. In other words, if $R$ is a ring of algebraic numbers, then $\mu(R)$ is the torsion subgroup of $R^\times$.  
The key ingredient which fuels all of our main theorems is that we can define roots of unity within rings of algebraic integers in $\Qtr(i)$.
The most general theorem of this flavor is the following.

\begin{theorem}[Theorem~\ref{thm:def_muL_in_OOL_general}]\label{thm:def_muL_in_OOL_general_intro}
    There is a single positive-existential formula 
    $\Tor(z)$ which defines $\mu(\OO_L)$ in $\OO_L$ for every field $L\subseteq \Qtr(i)$.
\end{theorem}

To prove this theorem, we first use the notion of interpreting formulas in quotient rings (see Lemma~\ref{lem:interpretations}) to reduce to the case of defining the torsion subgroup of units within a ring $A\subseteq \Qtr(i)$ of algebraic integers that contains a primitive $20$-th root of unity.
Then, we define a nonmaximal subring $R_{3}(A) := \ZZ + 3A$ inside $A$, following the strategy of \cite{Springer24}, so that $R_{3}(A)^\times$ is totally real.
By deploying several formulas throughout Section~\ref{sec:foundational_formulas} on this totally real set of units, we eventually obtain an existentially definable totally real subset containing only elements of the form $2 + \zeta + \zeta^{-1}$, where $\zeta$ is a root of unity. 
This lets us define many roots of unity as the roots of polynomials of the form $X^2  - (\zeta + \zeta^{-1})X + 1 $.
Using these roots of unity as a starting point, we then define a set containing precisely all roots of unity.

\subsection{The construction of Mazur, Rubin and Shlapentokh}
Mazur, Rubin and Shlapentokh presented a construction that, given a ring of integers $\OO_L$ and a subgroup of units $M\subseteq \OO_L^\times$, provides a subring $R_{\OO_L, M}\subseteq \OO_L$; see \cite{MRS24}.
Moreover, if $M$ is definable in $\OO_L$, then so is $R_{\OO_L, M}$.
We apply the construction with $M = \mu(\OO_L)$ and prove that $R_{\OO_L,\mu(\OO_L)} = \ZZ$, as long as $\mu(L)$ contains infinitely many roots of unity of prime order; see Theorem~\ref{thm:identify_R}.
Moreover, by mildly generalizing their construction and interpreting bounded-degree extension rings over a base ring, we prove the following, which is a stronger version of Theorem~\ref{thm:main_Zab}.

\begin{theorem}[Theorem~\ref{thm:def_Z_in_O_general}]\label{thm:def_Z_in_O_general_intro}
    Fix an integer $D\geq 1$ and let $\mathcal F_D$ be the set of fields $L\subseteq \Qtr(i)$ for which there are infinitely many roots of unity $\zeta_p$ of prime order $p$ satisfying the bound $[L(\zeta_p) : L]\leq D$. There is a single parameter-free first-order formula which defines $\ZZ$ in $\OO_L$ for every field $L\in \mathcal F_D$.
    Therefore, all such $\OO_L$ have undecidable first-order theory.
\end{theorem}

For comparison, recall that an algebraic extension $L$ of $\QQ$ is called \emph{big} if, for every positive integer $n$, there is a number field $F$ inside $L$ such that $n$ divides $[F : \QQ]$; see \cite{MR20}. The main result of Mazur, Rubin and Shlapentokh \cite[Thm.~4.8]{MRS24} states  that $\ZZ$ is first-order definable in the ring of integers of every algebraic extension of $\QQ$ which is not big. 
In~contrast, big fields like $\Qab$ are the central motivation of our results here.

\subsection{The blueprints of Robinson and Henson}

We can also deploy the definability of the torsion subgroup of units with a construction due to C.W. Henson \cite{vandenDries} which generalizes a construction of J. Robinson \cite{Robinson62}.
In particular, there is the following blueprint for proving undecidability for a ring of algebraic integers $A$: If there is a parameterized family of definable subsets of $A$ containing finite sets of arbitrarily large size, then there is an interpretation of R.M. Robinson's $\mathbf Q$, a weak essentially undecidable arithmetic theory, in the ring $A$.
Many results have been built on this well-known blueprint; see \cite{MRUV20, Shlapentokh18, Springer24}, to name just a few.
J. Robinson's version of the blueprint also underlies the concept of $\JR$-numbers, as discussed in the following subsection.

In Section~\ref{sec:Blueprint_Henson_and_Robinson}, our parameterized family of sets will contain finite sets corresponding to roots of unity of various prime-power orders.
If the ring in question does not include the necessary roots of unity, then we interpret extension rings of bounded degree over the base ring, which is sufficient for our purposes. 
Note that Theorem~\ref{thm:undec_intro} has less restrictive hypotheses than Theorem~\ref{thm:def_Z_in_O_general_intro} because the roots of unity $\zeta$ do not need to have prime order.

\begin{theorem}[Theorem~\ref{thm:undec}]
\label{thm:undec_intro}
    Let $L\subseteq \Qtr(i)$.
    If there exists a $D \geq 1$ so that $[L(\zeta) :L]\leq D$ for infinitely many roots of unity $\zeta$, then the first-order theory of $\OO_L$ is undecidable.
\end{theorem}

\subsection{Julia Robinson's numbers}

We remark that Theorem~\ref{thm:undec_intro} does not hypothesize about whether the field is totally real or totally imaginary.
To fit a format closer to previous theorems, like \cite[Theorem~1.4]{Springer24}, we can phrase the result in terms of totally imaginary quadratic extensions of totally real fields with a fixed $\JR$-number; see Section~\ref{sec:JR} for a definition and discussion of relevant background.

\begin{theorem}
\label{thm:undec_JR_intro}
	 Let $K$ be a totally real field for which $\JR(\OO_K) = 4$ is an attained minimum.
	 For every totally imaginary quadratic extension $L/K$, the first-order theory of $\OO_L$ is undecidable.
\end{theorem}

In particular, we observe that Theorem~\ref{thm:undec_JR_intro} completely subsumes the undecidability result \cite[Thm.~1.4]{Springer24}, although the existential definability of the totally real subring of integers \cite[Thm.~1.3]{Springer24} remains notable compared to the first-order (non-existential) definitions of $\ZZ$ in this paper.
We remark that Theorem~\ref{thm:undec_JR_intro} is also true if $K$ instead satisfies the condition $\JR(\OO_K) = \infty$; see Theorem~\ref{thm:undec_JR_infty}.

\subsection{AI declaration}
The author used ChatGPT 5.6 Sol to perform mathematical exploration and to obtain feedback on draft versions of this paper.
All arguments have been manually checked and written in the author's own language, and the author takes full responsibility for the material in this paper.

\section{Preliminaries}\label{sec:prelim}

To start, we define some notation and recall basic material.
As usual, all rings are commutative and contain $1$, subrings share $1$ with their ambient ring, and formulas are in the language of rings, $\mathcal L_{\text{ring}} = \{0,1,+,\cdot\}$.
Fix an algebraic closure $\Qbar$ of $\QQ$ and view algebraic extensions of $\QQ$ inside $\Qbar$.
For a subfield $L\subseteq \Qbar$, we write $\OO_L$ for the ring of all algebraic integers in $L$ and call $\OO_L$ the ring of integers of $L$.
If $A\subseteq \OO_L$ is a ring which is not necessarily maximal, then we simply call $A$ a ring of algebraic integers.
If $\zeta$ is a root of unity of order $n$, then we say $\zeta$ is a primitive $n$-th root of unity and write $\ord(\zeta) = n$.

We recommend  \cite{Washington} for general background on cyclotomic fields and CM fields.
If $K\subseteq \Qbar$ is totally real and $L/K$ is a totally imaginary quadratic extension, then there is a (unique) automorphism $\tau \in \Gal(L/K)$ such that $\sigma(\tau(x)) = \overline{\sigma(x)}$ for every embedding $\sigma : L\hookrightarrow \CC$ where $z\mapsto \overline z$ denotes complex conjugation in $\CC$.
In particular, if $L/\QQ$ is Galois, then $\tau$ is central in $\Gal(L/\QQ)$.
These facts in the case when $[L : \QQ] < \infty$ are well-known (see \cite[p.39]{Washington}), and the argument is precisely the same for the infinite-degree case.
As in the finite-degree case, we refer to the automorphism $\tau$ as complex conjugation and also write $\overline x:= \tau(x)$.

Write $\Qtr\subseteq \Qbar$ for the field of all totally real numbers and write $i = \sqrt{-1}$. We will use the condition that $L$ is a subfield of $\Qtr(i)$ as interchangeable with the condition that $L$ is either totally real or a totally imaginary quadratic extension of a totally real field.
Although this is seemingly well-known, we add a quick proof here for clarity.

\begin{lemma}
    If $L$ is an algebraic extension of $\QQ$, then $L$ is contained in $\Qtr(i)$ if and only if $L$ is either totally real or a totally imaginary quadratic extension of a totally real field. More precisely, if $L\subseteq \Qtr(i)$ is not totally real, then $L^+ := \Qtr \cap L$ is its maximal totally real subfield and $[L : L^+] = 2$.
\end{lemma}
\begin{proof}
    If $L$ is totally real, then $L\subseteq \Qtr$ by definition.
    If $L$ is a totally imaginary quadratic extension of a totally real field $L^+$, then $L = L^+(\sqrt{a})$ for $a\in L^+$ which is totally negative. But then $-a$ is totally positive, so $i\sqrt{a} = \sqrt{-a}\in \Qtr$. Hence $L\subseteq L^+(i, \sqrt{-a})\subseteq \Qtr(i)$.

    Conversely, let $L\subseteq \Qtr(i)$. 
    Because $\Qtr$ and $\QQ(i)$ are both Galois extensions of $\QQ$, it follows that $\Qtr(i)$ is as well. 
    Let $H = \Gal(\Qtr(i)/L)$.
    Because complex conjugation $\tau$, as described above, is central in the Galois group $\Gal(\Qtr(i)/\QQ)$,  we deduce $\tau H \tau^{-1} = H$.
    By the Galois correspondence, this means that 
    $$
        \tau(L) = \tau(\Qtr(i)^H) = \Qtr(i)^{\tau H \tau^{-1}}= \Qtr(i)^H = L.
    $$
    Therefore, $L$ is stable under complex conjugation.
    If $L$ is fixed pointwise by $\tau$, then it is totally real and we are done, so assume $x\neq \tau(x)$ for some $x\in L$.
    Because $\sigma(x) \neq \sigma(\tau(x)) = \overline{\sigma(x)}$ for every embedding $\sigma: L\hookrightarrow \CC$, we deduce that every embedding of $x$ is non-real. Thus,  $L$ is totally imaginary.
    Finally, by Galois theory, the subfield $L^+$ of $L$ fixed by $\tau$ satisfies $[L : L^+] =2$ and $L^+ = \Qtr \cap L$.
\end{proof}

Write $\Phi_n(X)\in \ZZ[X]$ for the $n$-th cyclotomic polynomial, namely the irreducible polynomial of degree $\varphi(n)$ whose roots are precisely the primitive $n$-th roots of unity.
We repeatedly use the following well-known lemma, so we list it explicitly.
\begin{lemma}
\label{le:cycl_norm}
If $\zeta$ is a primitive $n$-th root of unity and $n \neq 1$, then
$$
	| N_{\QQ(\zeta)/\QQ}(1 - \zeta)| = |\Phi_n(1)| = \begin{cases}
		p & \text{ if } n = p^a \text{ for some prime } p;\\
		1 & \text{ otherwise}.
	\end{cases}
$$
In particular, $1 - \zeta$ is a unit if and only if $n$ is not a prime power.
\end{lemma}

The following lemma is a useful characterization of the totally real elements that we will need to define prior to obtaining roots of unity themselves.
This characterization is also well-known and due to Kronecker \cite{Kronecker1857}. 
We refer  the reader to the proof provided in \cite[Lemma~2.1]{DKMWY25} for modern exposition.
\begin{lemma}\label{lem:traces_in_0_4}
    Let $\alpha$ be a totally real algebraic integer. If all conjugates of $\alpha$ lie in the interval $(0,4)$, then $\alpha = \zeta + \zeta^{-1} + 2$ for some root of unity $\zeta$.
\end{lemma}

Finally, we say that a subset $S\subseteq\Qtr(i)$ is stable under complex conjugation if $\overline x\in S$ whenever $x\in S$. We record the fact that extensions by roots of unity behave nicely with respect to complex conjugation. 

\begin{lemma}
    If $A$ is a ring of algebraic integers which is stable under complex conjugation and $\zeta$ is a root of unity, then $A[\zeta]$ is also stable under complex conjugation. In particular, $\OO_L[\zeta]$ is stable under complex conjugation for every $L\subseteq \Qtr(i)$ and every root of unity $\zeta$.
\end{lemma}
\begin{proof}
    This follows immediately from the fact that $\overline \zeta = \zeta^{-1} = \zeta^{\ord(\zeta) -1}\in A[\zeta]$.
\end{proof}

\section{Existential formulas and roots of unity}
\label{sec:foundational_formulas}

Throughout this section, let $A$ be a ring of algebraic integers in $\Qtr(i)$ which is stable under complex conjugation, and let $L$ be the field of fractions of $A$. 
As in previous work \cite{Springer24}, we need to use nonmaximal subrings to (temporarily) avoid roots of unity, because this lets us focus on the totally real elements.
Define the subring 
$$
    R_m(A) := \ZZ + mA.
$$
We observe that $R_m(A)$ and $R_m(A)^\times$ are both existentially definable in $A$. Indeed for $x\in A$:
\begin{align*}
    x\in R_m(A) &\iff  \bigvee_{a = 0}^{m-1} (\exists b \ x = a + mb);\\
    x\in R_m(A)^\times &\iff  x\in R_m(A) \wedge ( \exists u\in R_m(A), \ xu = 1).
\end{align*}

The following proposition gives useful properties of this subring $R_m(A)$.
The essence of this proposition is the same as \cite[\S2]{Springer24}, except that our prior work restricts to the case where $A = \OO_L$ is maximal. 
The proof in the more general setting is conceptually the same, and it is included to make this paper more self-contained.

\begin{prop}
	\label{prop:basic_Rm}
    If $m\geq 2$, then the following are true.
	\begin{enumerate}[(a)]
	    \item $\mu(R_m(A)) = \{\pm1\}$;
        \item Every $u\in R_m(A)^\times$ satisfies $\overline u = \pm u$;
        \item If $m\geq 3$, then every $u\in R_m(A)^\times$ satisfies $\overline u = u$.
	\end{enumerate}
\end{prop}

\begin{proof}
    First, if $\zeta$ is a  root of unity in $R_m(A)$, then $R_m(A)\supseteq \ZZ[\zeta]$. But also $R_m(A) \cap \QQ(\zeta) \subseteq R_m(\ZZ[\zeta])$. This implies $\ZZ[\zeta] = R_m(\ZZ[\zeta])$, which is only possible if $\zeta = \pm 1$ because $\ZZ[\zeta]$ is the ring of integers of $\QQ(\zeta)$. This proves the first claim; compare to \cite[Prop.~2.4]{Springer24}.

    Now let $u\in R_m(A)^\times$. The element $\zeta = u/ \overline u \in R_m(A)^\times$ has absolute value $1$ in every embedding because $L\subseteq \Qtr(i)$; see the preliminary material in Section~\ref{sec:prelim}. Thus, $\zeta$ is a root of unity. This proves that $u = \zeta \overline u$ for $\zeta\in \mu(R_m(A)) = \{ \pm 1\}$. This proof is essentially the same as the proof of the classic \cite[Thm.~4.12]{Washington}; see also \cite[Thm.~2.5]{Springer24}.

    Finally, assume $m\geq 3$ and $\overline u = -u$. Since $u\in R_m(A)$, we write $u = a + m\alpha$ for $a\in \ZZ$ and $\alpha\in A$. Then, $$
        2u = u- \overline u = (a + m\alpha) - (a + m\overline \alpha) = m(\alpha - \overline \alpha) \in mA
    $$
    This implies $2/m \in A$, which is impossible if $m  > 2$ because $A$ only includes algebraic integers. Therefore, $u = \overline u$, proving the final case.
\end{proof}

As a clarifying remark, we also mention the following lemma.
\begin{lemma}
    If $x$ and $y$ are algebraic integers with $xy = 1$, then $y\in \ZZ[x]$. Consequently, $A^\times\cap R_m(A) = R_m(A)^\times$.
\end{lemma}
\begin{proof}
    This argument is the same as \cite[Thm.~2.3.(a)]{Springer24}.
    Because $x$ is an algebraic unit, its minimal polynomial takes the form $f(X) = X^d + c_{d-1}X^{d-1} + \dots + c_1X \pm 1$. By rearranging the equality $f(x) = 0$, we find that 
    $$
        x(x^{d-1} + c_{d-1}x^{d-2} + \dots + c_1) = \mp 1
    $$
    We deduce that $y  = x^{-1}$ is an integral combination of powers of $x$, as desired. Therefore, if $x$ is a unit and is in $R_m(A)$, then so is its inverse.
\end{proof}

Now consider the formula $V(t)$ defined by
\begin{equation}
	\exists u \in R_3(A)^\times, \ (u-1) t = 2u + 1.
\end{equation}

We may conceptualize the set of $t\in A$ for which $V(t)$ holds as the intersection of $A$ and the image of a fractional-linear transformation applied to $R_3(A)^\times$.
Because this transformation is invertible, we observe that there is a fractional-linear transformation $T$ mapping any $t$ to a potential witness $u = T(t)$, which is truly a witness if $T(t) \in R_3(A)^\times$.
We constructed $T$ so that $T(t)\in R_3(A)^\times$ for certain desired totally real $t = \zeta + \zeta^{-1}$; see the lemmas below.

\begin{lemma}
	\label{lem:V_real}
	If $t\in A$ and $V(t)$ holds, then $t$ is totally real.
\end{lemma}
\begin{proof}
	Suppose $(u-1) t = 2u + 1$ for $u\in R_3(A)^\times$.
	Clearly, $u \neq 1$ because $0\neq 3$. 
	Therefore, 
	$
	t = \frac{2u+1}{u-1}
	$
	is totally real because $u$ is totally real by Proposition~\ref{prop:basic_Rm}.
\end{proof}

\begin{lemma}
	\label{lem:trace_in_V}
	Let $\zeta\in A$ be a primitive $n$-th root of unity.
	If $\frac{n}{\gcd(n,3)}$ is greater than 1 and not a prime power, 
	then $V(\zeta + \zeta^{-1})$ holds in $A$.
\end{lemma}
\begin{proof}
	Define $t = \zeta + \zeta^{-1}$ and $u = \frac{t +1}{t - 2}$. We will show $u\in R_3(A)^\times = A^\times \cap R_3(A)$, which proves $V(t)$ holds by definition after inverting the fractional-linear transformation.
	
	Because $\zeta^3$ is a root of unity of order $\frac{n}{\gcd(n,3)}$, 
	the hypotheses imply that both $1 - \zeta$ and $1 - \zeta^3$ are units; see Lemma~\ref{le:cycl_norm}.
	Specifically, both of the following are units:
	\begin{align*}
	t -2 &= \zeta^{-1}(1 - \zeta)^2 &  t+1&= \zeta^{-1} \frac{1 - \zeta^3}{1 - \zeta}.
	\end{align*}
	Therefore, $u$ is a unit.
	Moreover, because $t -2 \equiv t + 1\bmod 3$, we conclude that $u \in 1 + 3A \subseteq R_3(A)$ by definition, establishing $u\in R_3(A)^\times$ as desired.
\end{proof}

We have now established that $V(t)$ defines a totally real subset that contains $\zeta + \zeta^{-1}$ for certain roots of unity $\zeta$.
To ensure that we ultimately define a totally real set with no extraneous unwanted elements, we shift our attention to elements of the form $\zeta + \zeta^{-1} + 2$, which are precisely the elements described by Kronecker's theorem; see Lemma~\ref{lem:traces_in_0_4}.
In particular, we define the following formulas.
\begin{align}
	B(x) &\iffdef (\exists s \ V(s) \wedge x = s^2 ) \wedge ( \exists r \  V(r) \wedge x = 4-r^2);\\
	M(z) &\iffdef \exists x \ (B(x) \wedge (z^2 - (x - 2)z + 1 = 0)).\label{eq:M}
\end{align}

\begin{lemma}
\label{lem:B_real}
If $x\in A$ and $B(x)$ holds, then $x = \zeta + \zeta^{-1} + 2$ for some root of unity $\zeta$.
\end{lemma}
\begin{proof}
	The elements $r,s$ which witness the fact that $B(x)$ holds are totally real by Lemma~\ref{lem:V_real}. 
	Consequently, $x = s^2$ and $4-x = r^2$ are both totally nonnegative, proving all conjugates of $x$ lie in the closed interval $[0,4]$. Note that $x\in \{0,4\}$ implies that either $s$ or $r$ is zero.
    But $V(0)$ does not hold, since this implies $2u  =-1$, which contradicts integrality. 
    This implies $x = \zeta + \zeta^{-1} +2$; see Lemma~\ref{lem:traces_in_0_4}.
\end{proof}

\begin{lemma}
\label{lem:all_M_is_zeta}
	 If $z\in A$ and $M(z)$ holds, then $z$ is a root of unity.
\end{lemma}
\begin{proof}
    If $x$ is a witness of $M(z)$, then Lemma~\ref{lem:B_real} implies $z$ is a root of
    $$
        T^2 - (x - 2)T +1 = T^2 - (\zeta + \zeta^{-1})T + 1 = (T-\zeta)(T - \zeta^{-1}),
    $$
    where $\zeta$ is a root of unity. Hence, $z = \zeta$ or $z = \zeta^{-1}$.
\end{proof}

Now we are almost done.
Indeed, we have a formula $M(z)$ which defines a subset containing only roots of unity.
The following lemma establishes the fact that $M$ holds for ``enough'' roots of unity, which enables the definability of the full torsion subgroup in the following section.

\begin{lemma}
\label{lem:M_holds_for_zeta^2}
Let $m\geq1$ be an integer for which $\frac{m}{\gcd(m,3)}$ is greater than $1$ and not a power of $2$. 
Assume that $A$ contains $i = \sqrt{-1}$.
If $\zeta\in A$ is a primitive $m$-th root of unity, then $M(\zeta^2)$ holds in $A$.
\end{lemma}
\begin{proof}
    Choose $\epsilon \in \{\pm1\}$ so that both $\eta := \epsilon\zeta$ and $\xi := i\epsilon \zeta$ have even order.
    To see that this is possible, consider the $2$-primary part $\alpha$ of $\zeta$.
    The elements of the set $\{\alpha, -\alpha, i\alpha, -i\alpha\}$ are all distinct roots of unity, so at most one of them is equal to one.
    Thus, we simply choose $\epsilon$ so that both $\epsilon \alpha$ and $\epsilon i \alpha$ are nontrivial.

    Now we have $\eta^2 = \zeta^2$ and $\xi^2 = -\zeta^2$ for roots of unity $\eta$ and $\xi$, whose orders are even. 
    Because $\frac{m}{\gcd(m,3)}$ is not a power of 2, this implies that Lemma~\ref{lem:trace_in_V} applies to $\eta$ and $\xi$, establishing that $V(s)$ and $V(r)$ hold in $A$ for $s = \eta + \eta^{-1}$ and $r = \xi + \xi^{-1}$.
	Writing $x = 2 + \zeta^2 + \zeta^{-2}$, we have
\begin{align*}
	s^2 &= \eta^2 + 2 + \eta^{-2} = \zeta^2 + 2 + \zeta^{-2} = x\\
	4-r^2 &= 4 - (\xi^2 + 2 + \xi^{-2}) = 2 + \zeta^2 + \zeta^{-2} = x.
\end{align*}
Thus, $B(x)$ holds and $\zeta^2$ is a root of $z^2 - (x-2)z + 1$, so $M(\zeta^2)$ holds by definition.
\end{proof}

\section{Defining the torsion subgroup of units}\label{sec:def_tor}

\subsection{Rings containing certain roots of unity}
\label{subsec:ring_with_zeta}
We now seek to obtain a definition of the torsion subgroup of units.
First, we concern ourselves with the particular case of rings which contain certain roots of unity. 
The following theorem is straightforward from the formulas we developed in Section~\ref{sec:foundational_formulas}, and we will generalize it in Theorem~\ref{thm:def_muL_in_OOL_general}. 

Define the formula
	\begin{equation}\label{eq:def_muL_if_20}
		\widetilde\Tor(z) \iffdef \ \exists \omega \ (\omega^{5} = 1 \wedge M(\omega z^2))
	\end{equation}
    
\begin{theorem}
    \label{thm:def_muL_in_OOL_if_20th}
	If $A$ is a ring of algebraic integers in  $\Qtr(i)$ which is stable under complex conjugation and contains a primitive $20$-th root of unity, then $\widetilde\Tor(z)$ holds for $z\in A$ if and only if $z$ is a root of unity.
\end{theorem}
\begin{proof}
    If $\widetilde\Tor(z)$ holds with witness $\omega$, then $\omega$ is a root of unity by definition, and $\omega z^2$ is a root of unity by Lemma~\ref{lem:all_M_is_zeta}. 
    Hence $z$ is a root of unity too.

    Conversely, suppose $z\in \mu(A)$.
    Choose a $5$-th root of unity $\omega \in A $ so that $\omega^3z$ has order divisible by $5$. Indeed, we can simply choose $\omega = 1$ if $z$ already has order divisible by $5$, and choose $\omega$ to be a primitive $5$-th root of unity otherwise.
    Then, by construction, the order $m = \ord(\omega^3z)$ is greater than $1$ and $\frac{m}{\gcd(m,3)}$ is not a power of $2$. Moreover, $A$ contains $i = \sqrt{-1}$ because it contains a primitive $20$-th root of unity. Thus, Lemma~\ref{lem:M_holds_for_zeta^2} applies, and we deduce that $M((\omega^3z)^2) = M(\omega z^2)$ holds.
\end{proof}

\begin{remark}
    The prime $5$ is merely convenient in Theorem~\ref{thm:def_muL_in_OOL_if_20th}. 
    Clearly, we could replace that prime with any odd prime power $m$ greater than $3$, assuming instead that $A$ contains a primitive $4m$-th root of unity.
    Other variants on this theme are also possible.
    Rather than write the theorem in a more cumbersome form, we are content to write the given version, which minimizes the degree of the auxiliary polynomial, and reserve greater generality for the following subsection.
\end{remark}

\subsection{Definability of the torsion subgroup of units in general}
The preceding section contains an extraneous hypothesis, namely assuming that the rings in question contain a given useful root of unity.
To circumvent this requirement, we use the standard technique of interpreting extension rings over a base ring, as in the following lemma.
This material is standard; see \cite[Chapter 5]{Hodges}. 
We are also inspired by work of Shlapentokh; see \cite{Shlapentokh18}, for example.
We apply this trick to obtain more general results with regard to both decidability and definability.

\begin{lemma}
    \label{lem:interpretations}
    Fix $d\geq 1$. Let $R$ be a ring. For ${\bf c} \in R^d$, define
    $$
        f_{{\bf c}} = X^d + c_{d-1}X^{d-1} + \dots + c_0.
    $$
    The $R$-algebra $R_{{\bf c}} = R[X]/(f_{{\bf c}}(X))$ is uniformly first-order interpretable in $R$ with domain $R^d$ and parameters ${\bf c}$.
    In particular, for every fixed $\mathcal L_{\text{ring}}$-formula $\psi$ applied to $R_{\bf c}$ there is an $\mathcal L_{\text{ring}}$-formula $\psi^{[d]}$ over $R$ that defines the coordinate preimage of $\psi$ uniformly in ${\bf c}$. 
    Each parameter in $\psi$ corresponds to $d$ parameters in $\psi^{[d]}$. 
    If $\psi$ is (positive)-existential, then $\psi^{[d]}$ is the same.
\end{lemma}
\begin{proof}
Identifying $R_{{\bf c}} = R[X]/(f_{{\bf c}}(X)) \cong R^d$, addition and equality are both coordinate-wise, while multiplication is given by polynomial multiplication and the rule that $X^d = - c_{d-1}X^{d-1} - \dots - c_0$. Translation of arbitrary formulas follows by induction.
\end{proof}

\begin{theorem}\label{thm:def_muL_in_OOL_general}
    There is a single parameter-free positive-existential formula 
    $\Tor(z)$ which defines $\mu(\OO_L)$ in $\OO_L$ for every subfield $L\subseteq \Qtr(i)$.
\end{theorem}
\begin{proof}
    For each $d\geq 1$, translate the formula \eqref{eq:def_muL_if_20} by the interpretation trick of Lemma~\ref{lem:interpretations} to get a formula over $R^d \cong R[X]/(f_{\bf c}(X))$ for any ring $R$, with parameters ${\bf c}$ for the coefficients of the monic polynomial $f_{\bf c}$.
     Write $\Tor^{(d)}(z; {\bf c})$ for the result of evaluating this formula on elements of the form $(z, 0,\dots, 0)$, and define
    $$
        \Tor(z)
         \iffdef  
            \bigvee_{d = 1}^8 \left(\exists c_0\dots \exists c_{d-1} \ (f_{\bf c}\mid \Phi_{20}
            \wedge \Tor^{(d)}(z;{\bf c})\right)
    $$
    Above, we use $(c_0,\dots, c_{d-1}) = {\bf c}$ as shorthand. 
    The divisibility condition $f_{\bf c}\mid \Phi_{20}$ in $R[X]$ corresponds to a positive existential formula which quantifies the coefficients of a polynomial $h$ of degree $8-d$ and imposes $hf_{\bf c} = \Phi_{20}$.

    Fix a field $L\subseteq \Qtr(i)$. We will show that $\Tor$ defines $\mu(\OO_L)$, as desired.
    First, suppose that $z\in \OO_L$ is a root of unity.
    Fix a primitive $20$-th root of unity $\theta$ and let $g$ be its minimal polynomial over $L$.
    Say $g$ has degree $d$ and observe that $d = [L(\theta) : L] \leq [\QQ(\theta) :  \QQ] = 8$.
    Because $\OO_L$ is integrally closed, the tuple ${\bf c}$ of coefficients of $g = f_{\bf c}$ lie in $\OO_L^d$.
    Moreover $g$ must divide the minimal polynomial of $\theta$ over $\QQ$, which is $\Phi_{20}(X)$ by definition.
    Then, $\OO_L[\theta]\cong \OO_L[X]/(f_{\bf c}(X))$ and $\Tor(z)$ holds with our chosen witness ${\bf c}$ by Theorem~\ref{thm:def_muL_in_OOL_if_20th}.

     Conversely, suppose that $\Tor(z)$ holds with witness ${\bf c}\in \OO_L^d$, and let $g(X)$ be a monic irreducible factor over $L$ of $f_{\bf c}(X)$.
     Because $f_{\bf c}(X)\mid \Phi_{20}(X)$, we observe that $g(X)$ is the minimal polynomial for some primitive $20$-th root of unity, call it $\theta$, and $g$ has coefficients in $\OO_L$. 
     Observe that there is a homomorphism 
     $$
        \OO_L[X]/(f_{\bf c}(X))\to \OO_L[X]/(g(X)) \cong \OO_L[\theta].
    $$
     Because positive-existential formulas are preserved under homomorphism, we deduce that the formula also holds for the image of $z$ under this homomorphism, which means that $z$ must be a root of unity by another application of Theorem~\ref{thm:def_muL_in_OOL_if_20th}.
\end{proof}

\begin{remark}
    The following is slightly subtle, so we point it out explicitly. Theorem~\ref{thm:def_muL_in_OOL_if_20th} applies to nonmaximal subrings $A\subseteq \OO_L$ because our machinery in Section~\ref{sec:foundational_formulas}  does not care about maximality.
    However, our specific deployment of the interpretation trick of Lemma~\ref{lem:interpretations} in the proof of Theorem~\ref{thm:def_muL_in_OOL_general} asserts that our ring contains the coefficients of the monic minimal polynomial of $\theta$, a primitive 20-th root of unity. 
    This is always true for the integrally closed $\OO_L$, but not necessarily for $A\subseteq \OO_L$ in general. 
    Of course, the theorem can be slightly generalized with care, but we prioritize the clean version above.
\end{remark}

\subsection{Subsets of the torsion subgroup of units}
We finish the section by showing how formulas defining $\mu(A)$ in $A$ can be used to define the specific subset of roots of unity which have prime-power order. 
Note that the property of two elements of $A$ being associates (i.e., unit multiples) is positive-existentially definable:
$$
	\Assoc(a,b) \iffdef \exists \epsilon \exists \delta \ (\epsilon\delta = 1 \wedge b = \epsilon a).
$$
This allows us to define the subset that we seek.

\begin{lemma}
\label{lem:assoc_iff}
    Let $m = p^e$ be a prime power and let $\zeta\in A$ be a primitive $m$-th root of unity.
    If $z\in A$ is any root of unity, then $1-z$ is an associate of $1 - \zeta$ if and only if $\ord(z) = m$.
\end{lemma}
\begin{proof}
    First, assume that $z$ is a primitive $m$-th root of unity.
    In particular, we have $z = \zeta^a$ for some $1\leq a< m$ with $\gcd(a,m) = 1$, and it is well known that the cyclotomic integer $\frac{1 - z}{1 - \zeta} = \frac{\zeta^a - 1}{\zeta - 1} = 1 + \zeta + \dots + \zeta^{a-1}$ is a unit; see \cite[\S8.1]{Washington}, for example.
    This completes the first direction.

    Conversely, let $n = \ord(z)$ and assume that $1-z$ and $1 - \zeta$ are associates. 
    Define $E = \QQ(z,\zeta)$.
    Because $1 -z$ and $1-\zeta$ are associates, we find that
    $$
        |\Phi_n(1)^{[E : \QQ(z)]} |
            = |N_{E/\QQ}(z - 1)| 
            = |N_{E/\QQ}(\zeta-1) |
            = |\Phi_m(1)^{[E : \QQ(\zeta)]}|
            = p^{[E : \QQ(\zeta)]}
    $$
    By Lemma~\ref{le:cycl_norm}, $n = p^f$ is a power of $p$ and $[E : \QQ(z)] = [E : \QQ(\zeta)]$, implying 
    $$
        \varphi(p^f) = [\QQ(z) : \QQ] = [\QQ(\zeta) : \QQ] = \varphi(p^e).
    $$
    We conclude that $n = m$, as desired.
\end{proof}

Define the following.
\begin{align}
    \PrPo(z) &\iffdef z\in \mu(A) \wedge z \neq 1 \wedge (1-z)\not\in A^\times;\\
    \Level(z,w) &\iffdef \PrPo(z)\wedge \PrPo(w) \wedge \Assoc(1 -z, 1-w);\\
    \psi(z;w) &\iffdef z\in \mu(A) \wedge \Assoc(1-z,1-w).\label{eq:psi_def}
\end{align}
We use $\mu(A)$ as shorthand in the definitions above, noting that all of these constitute first-order formulas if $\mu(A)$ is definable in $A$.
In fact, $\psi$ is positive-existential when using the positive-existential definitions for $\mu(A)$ given in Theorems~\ref{thm:def_muL_in_OOL_if_20th} and \ref{thm:def_muL_in_OOL_general}, while the other displayed formulas are not existential.
Combining the preceding results, we deduce the following characterization of elements which are roots of unity of prime-power order.

\begin{theorem}\label{thm:prime_power_if_20th}
Let $A$ be a ring of algebraic integers. The following are true.
\begin{enumerate}[(a)]
    \item  $\PrPo(z)$ holds if and only if $z$ has prime-power order.
    \item $\Level(z,w)$ holds if and only if $z$ and $w$ are both roots of unity of prime-power order and $\ord(z) = \ord(w)$.
    \item If $w$ has prime-power order $m > 1$, then $\psi(z;w)$ holds precisely when $z$ is also a primitive $m$-th root of unity.
\end{enumerate}
Moreover, if either of the following additional conditions hold, then $\PrPo$ and $\Level$ correspond to first-order formulas, and $\psi$ gives a positive-existential formula.
\begin{enumerate}[(I)]
    \item $A$ is a subring of $\Qtr(i)$ which is stable under complex conjugation and contains a primitive $20$-th root of unity;
    \item $A = \OO_L$ for a subfield $L\subseteq \Qtr(i)$.
\end{enumerate}
\end{theorem}
\begin{proof}
    The first claim follows from Lemma~\ref{le:cycl_norm}. The second and third follow from Lemma~\ref{lem:assoc_iff}.
    As anticipated, the final claims are direct applications of Theorems~\ref{thm:def_muL_in_OOL_if_20th} and \ref{thm:def_muL_in_OOL_general}.
\end{proof}

\section{Deploying the construction of Mazur, Rubin and Shlapentokh}
We now seek to exploit the definability of the torsion subgroup of units.
In this section, we will use the torsion subgroup with a recent construction due to Mazur, Rubin and Shlapentokh \cite{MRS24}, which we generalize slightly.
Indeed, given a commutative ring $A$ and a multiplicative monoid $M\subseteq A$, we define
$$
    R_{A,M} := \{x\in A : \forall \epsilon \in M \ \exists \delta \in M \ \exists y\in A, \ \delta -1 = (\epsilon - 1)x + (\epsilon - 1)^2y\}.
$$
Compare to \cite[Def.~2.1]{MRS24}, which only considers the case of $A = \OO_L$ for an algebraic extension $L$ of $\QQ$. 
The following is essentially a recapitulation of \cite[Lem.~2.2 and Prop.~2.3]{MRS24}, presented in the more general context we desire here. 
Notably, the wider generality is convenient because we will apply the construction to extension rings which may fail to be integrally closed, or even an integral domain.
Although the proof strategy is not new, we include the proof to make this paper more self-contained.

\begin{prop} \label{prop:general_R_AM}
Fix a commutative ring $A$ of characteristic 0 and a multiplicative monoid $M\subseteq A$. The following are true.
\begin{enumerate}[(a)]
    \item If $M$ is first-order definable in $A$, then $R_{A,M}$ is first-order definable in $A$.
    \item $R_{A,M}$ contains $0$ and $1$, and is closed under both addition and multiplication. 
    \item If $M$ is a subgroup of $A^\times$, then $R_{A,M}$ is a subring satisfying $\ZZ\subseteq R_{A,M}\subseteq A$.
\end{enumerate}
\end{prop}
\begin{proof}
    The definability claim is immediate. For the second statement, we note that $0\in R_{A,M}$ is witnessed by $\delta = 1$ for all $\epsilon$, and $1\in R_{A,M}$ is witnessed by $\delta = \epsilon$. To show closure under addition and multiplication, fix $x_1,x_2\in R_{A,M}$ and choose any $\epsilon \in M$. 
    By definition, for $i\in \{1,2\}$, there exists $\delta_{x_i}\in M$ so that 
    \begin{align*}
        \delta_{x_i} - 1 &\equiv (\epsilon - 1)x_i\bmod (\epsilon - 1)^2
    \end{align*}
    Then, we observe that $\delta_{x_1 + x_2} := \delta_{x_1}\delta_{x_2}$ suffices for the element $x_1 + x_2$ by the following computation:
    \begin{align*}
        (\delta_{x_1}\delta_{x_2}) - 1 &= (\delta_{x_1} - 1) + (\delta_{x_2} -1) + (\delta_{x_1} - 1)(\delta_{x_2} - 1)\\
        &\equiv(\epsilon - 1)x_1 + (\epsilon -1)x_2 \bmod (\epsilon -1)^2
    \end{align*}
    This implies that $R_{A,M}$ is closed under addition.
    
    For multiplication, apply the definition of $R_{A,M}$ to $x_2$ relative to $\epsilon' = \delta_{x_1}$ to get $\delta_{x_2}'\in M$ satisfying
    $$
        \delta_{x_2}' - 1 \equiv (\delta_{x_1} - 1)x_2 \equiv (\epsilon -1 )x_1x_2 \bmod ( \epsilon -1)^2.
    $$
    The first equivalence above uses the fact that $(\delta_{x_1} - 1)^2\equiv 0 \bmod (1 - \epsilon)^2$ by the definition of $\delta_{x_1}$.
    This shows that the product $x_1x_2$ is in $R_{A,M}$, completing the second claim.

    To prove the final claim, it is enough to prove that $-1\in R_{A,M}$, which implies that the semiring is indeed a ring.
    For each $\epsilon\in M$, take $\delta = \epsilon^{-1}$. This is possible when $M$ is a group. Then,
    \[
        \epsilon^{-1} -1 = -(\epsilon - 1) + \epsilon^{-1}(\epsilon - 1)^2 \equiv -(\epsilon -1) \bmod (\epsilon - 1)^2,
    \] 
    proving that $-1\in R_{A,M}$, as desired.
\end{proof}

Our first goal is to merely identify which ring is obtained when applying this construction with a torsion subgroup of units, as presented in Theorem~\ref{thm:identify_R}. Although our definability of the torsion subgroup in Section~\ref{sec:def_tor} puts restrictions on the fraction field of $A$, we require no such restrictions for this identification of $R_{A,\mu(A)}$. 
We start with two lemmas.

\begin{lemma}\label{lem:div_by_p}
     Let $A$ be any ring of algebraic integers containing a root of unity $\zeta_p$ of prime order $p$, and let $x\in A$ be an element with monic minimal polynomial $f(T)$ over $\ZZ$. 
     If $x \in \ZZ + (\zeta_p - 1)A$ and $\deg(f) \geq 2$, then the discriminant $\disc(f)$ is divisible by $p$.
\end{lemma}
\begin{proof}
    Let $F$ be a finite Galois extension of $\QQ$ containing both $x$ and $\zeta_p$, and let $x = a + (\zeta_p-1)y$ for $a\in \ZZ$ and $y\in A$.
    If $\deg(f) \geq 2$, then there is an automorphism $\sigma \in \Gal(F/\QQ)$ so that $\sigma(x) \neq x$.
     It is well-known that $ u = \frac{\sigma(1 -\zeta_p)}{1-\zeta_p}$ is a unit; see \cite[Chapter~8]{Washington}.
	 Therefore, because integers are fixed under automorphisms,
	 $$
	 	\sigma(x) = \sigma( a + (\zeta_p-1)y) = a + (\zeta_p - 1)u\sigma(y).
	 $$
	 Taking the difference, we find $x - \sigma(x) \in (\zeta_p -1)\OO_F$. 
     Therefore, using the definition of discriminant, we deduce that $\disc(f)$ is an integer contained in $(x - \sigma(x))^2\OO_F \subseteq (1 - \zeta_p)\OO_F$.
     Because $(1 - \zeta_p)\OO_F\cap \ZZ = p\ZZ$, we deduce that $\disc(f)$ is divisible by $p$.
\end{proof}

\begin{lemma}\label{lem:intersection_cases}
    Let $A$ be any ring of algebraic integers, and let $S$ be a set of primes $p$ for which $A$ contains a primitive $p$-th root of unity $\zeta_p$.
    $$\bigcap_{p\in S} (\ZZ + (\zeta_p -1)A) = \begin{cases}
                \ZZ + \prod_{p\in S} (\zeta_p - 1)A & \text{ if } \#S < \infty;\\
                \ZZ & \text{ otherwise.}
                \end{cases}
    $$
\end{lemma}
\begin{proof}
    It is clear that the left-hand side contains the right-hand side, so we only concern ourselves with the opposite containment.
     When $S$ is finite, equality follows from the Chinese Remainder Theorem because $(\zeta_p - 1)A \cap \ZZ = p\ZZ$ are pairwise comaximal ideals as $p\in S$ varies; see Lemma~\ref{le:cycl_norm}.
     Indeed, if $a_p\in \ZZ$ are integers so that $x\equiv a_p \bmod (\zeta_p -1)A$ for each $p\in S$, then there is a single $a\in \ZZ$ so that $a \equiv a_p\bmod p\ZZ$ for all $p\in S$.
     Then $x - a\in \cap_{p\in S} (\zeta_p - 1)A = \prod_{p\in S} (\zeta_p - 1)A$.

    If $S$ is infinite and $x\in \cap_{p\in S} (\ZZ + (\zeta_p -1)A)$ is not an integer, then the monic minimal polynomial of $x$ has discriminant divisible by infinitely many primes by Lemma~\ref{lem:div_by_p}. 
    This is impossible, so it is clear that the intersection is precisely $\ZZ$.
\end{proof}

\begin{theorem}
	\label{thm:identify_R}
    Let $A$ be any ring of algebraic integers.
    Writing $S(A)$ for the set of primes $p$ such that $A$ contains a primitive $p$-th root of unity $\zeta_p$, we determine
    $$
        R_{A,\mu(A)} = \bigcap_{p\in S(A)} (\ZZ + (\zeta_p -1)A) = \begin{cases}
                \ZZ + \prod_{p\in S(A)} (\zeta_p - 1)A & \text{ if } \#S(A) < \infty;\\
                \ZZ & \text{ otherwise.}
        \end{cases}
    $$
\end{theorem}
\begin{proof}
    Fix an element $x\in R_{A,\mu(A)}$ and $p\in S(A)$.
    We will show $x\in \ZZ + (\zeta_p - 1)A$. 
    Choose $\epsilon = \zeta_p$ and let $\delta,y$ satisfy
     \begin{equation}\label{eq:construction_of_RLmuL}
        \delta - 1 = (\epsilon - 1)x  + (\epsilon - 1)^2y.
     \end{equation}
     If $\delta = 1$, then $x = -(\epsilon - 1)y\in (\zeta_p - 1)A$, so there is nothing to show.
     Thus, we may assume $\delta \neq 1$ and define $E = \QQ(\zeta_p, \delta,x,y)$. 
	 By Lemma~\ref{le:cycl_norm}, $|N_{E/\QQ}(\zeta_p -1)| = p^{[E : \QQ(\zeta_p)]}$ because $\zeta_p$ is a primitive $p$-th root of unity. 
	 Because $\delta -1 = (\zeta_p - 1)x \bmod (\zeta_p - 1)^2\OO_E$, we deduce that $\frac{\delta - 1}{\zeta_p -1}\in \OO_E$.
     Hence, $N_{E/\QQ}(\delta - 1)$ is divisible by $p^{[E : \QQ(\zeta_p)]}$. 
	 Applying Lemma~\ref{le:cycl_norm} to $\delta$, we conclude this is only possible if $\ord(\delta) = p^f$ is a power of $p$ and $[E : \QQ(\delta)] \geq [E : \QQ(\zeta_p)]$. 
	 But this implies that 
     $$p^{f-1}(p-1) = [\QQ(\delta) : \QQ]\leq [\QQ(\zeta_p) : \QQ] = p-1.$$ 
	 Therefore, $f = 1$ and $\ord(\delta) = p$, which means that $\delta = \zeta_p^a$ for some $0< a < p$.
     Now we simply divide \eqref{eq:construction_of_RLmuL} by $\zeta_p - 1$ to find:
     $$
        x \equiv \frac{\delta - 1}{\zeta_p - 1} 
        \equiv \frac{\zeta_p^a - 1}{\zeta_p - 1} 
        \equiv 1 + \dots + \zeta_p^{a-1}
        \equiv a \bmod (1-\zeta_p)A.
     $$

     Conversely, suppose that $x \in \cap_{p\in S(A)} (\ZZ + (1 - \zeta_p)A)$.
     Consider any $\epsilon \in \mu(A)$.
     The case $\epsilon = 1$ is trivial to verify, so assume $\epsilon \neq 1$.
     Some power of $\epsilon$ has prime order, say $\epsilon^e = \zeta_p$ for $p\in S(A)$ and $e\geq 1$. 
     This implies that $\zeta_p -1 = \epsilon^e - 1 = (\epsilon - 1)(\epsilon^{e-1} +\dots + 1)$ is a multiple of $\epsilon -1$.
     In particular, $x = a + (\epsilon - 1)b$ for some $a\in \Z$ and $b\in A$ because $x\in \ZZ + (\zeta_p - 1)A \subseteq \ZZ + (\epsilon -1)A$ by assumption.
     Moreover, because $p\in (\zeta_p - 1)A \subseteq (\epsilon - 1)A$, we can assume $a \geq 0$ without loss of generality by replacing $a$ by some $a + pk$ if necessary, adjusting $b$ accordingly.
     Then, by the binomial theorem, choosing $\delta = \epsilon^{ a}$ provides
     $$
        \delta - 1 
            = (\epsilon - 1 + 1)^a - 1
             = \sum_{k = 1}^a \binom{a}{k} (\epsilon - 1)^k \equiv a(\epsilon -1) \equiv (\epsilon - 1)x \bmod (\epsilon - 1)^2A.
     $$
     Thus, $x\in R_{A,\mu(A)}$.
    The final equality is given by Lemma~\ref{lem:intersection_cases}.
\end{proof}

This theorem immediately provides a definition of $\ZZ$ in certain rings of algebraic integers.
\begin{corollary}
	\label{cor:def_Z_in_OOL}
	There is a single parameter-free first-order formula which defines $\ZZ$ in $A$ whenever $A$ is a ring of algebraic integers in $\Qtr(i)$ which is stable under complex conjugation and contains both a primitive $20$-th root of unity and a primitive $p$-th root of unity for infinitely many primes $p$.
\end{corollary}
\begin{proof}
	 Under the hypotheses on $A$, we observe that $\mu(A)$ is definable in $A$ by Theorem~\ref{thm:def_muL_in_OOL_if_20th}, proving $R_{A, \mu(A)} = \ZZ$ is definable in $A$ via Proposition~\ref{prop:general_R_AM}.
\end{proof}

Now we repeat the argument with some extra care to obtain our more general theorem.
\begin{theorem}\label{thm:def_Z_in_O_general}
    Fix an integer $D \geq1 $ and let $\zeta_n$ be a primitive $n$-th root of unity for all $n$. There is a single parameter-free first-order formula $Z_D(X)$ which defines $\ZZ$ in $\OO_L$ for every field $L\subseteq \Qtr(i)$ such that the set
    $$
        S_D := \{p \text{ prime} : [L(\zeta_p) : L] \leq D\}
    $$
    is infinite. 
    Therefore, all such $\OO_L$ have undecidable first-order theory.
\end{theorem}
\begin{proof}
    For each $1\leq d\leq 8D$ and ring $R$, define the following.
    Let $A_{\bf c} = R[X]/(f_{\bf c}(X))$ for ${\bf c}\in R^d$ and let $\theta_{\bf c}$ denote the class of $X$ in $A_{\bf c}$.
    Let $\widetilde \Tor^{(d)}(z; {\bf c})$ be the interpretation\footnote{We retain the tilde $\widetilde{\cdot}$ in the notation because we apply this formula to the entirety of $A_{\bf c}$, unlike the interpreted formula in the proof of Theorem~\ref{thm:def_muL_in_OOL_general} itself.} of formula \eqref{eq:def_muL_if_20} in $A_{\bf c}$ by Lemma~\ref{lem:interpretations}, 
    and let $M_{\bf c}$ be the subset of $A_{\bf c}$ defined by $\widetilde \Tor^{(d)}(z; {\bf c})$.
    Let $\Group^{(d)}(\bf c)$ be the first-order formula stating that the set $M_{\bf c}$ is a subgroup of $A_{\bf c}^\times$.
    
    With this notation, we define the formula:
    $$
        Z_D(x) \iffdef \bigwedge_{d = 1}^{8D} \forall {\bf c} 
        \left(\Group^{(d)}({\bf c}) \to x\in R_{A_{\bf c}, M_{\bf c}}
        \right).
    $$
    Note that $Z_D(x)$ corresponds to a first-order formula because $M_{\bf c}$ is uniformly defined in $A_{\bf c}$ by $\widetilde \Tor^{(d)}(z;{\bf c})$, which implies that there is a single formula defining $R_{A_{\bf c}, M_{\bf c}}$ in $A_{\bf c}$ with parameters ${\bf c}$.

    To see that this formula performs as desired, we first observe that Proposition~\ref{prop:general_R_AM} shows that $R_{A_{\bf c}, M_{\bf c}}$ is a ring containing $\ZZ$ whenever $M_{\bf c}$ is a group, and this is guaranteed when $\Group^{(d)}({\bf c})$ holds.
    Thus, $Z_D(n)$ holds for all integers $n\in \ZZ$.

    Now suppose that $x\in \OO_L$ is a non-integer so that $Z_D(x)$ holds.
    Consider a prime $p\in S_D$ with $p > 5$.
    Let ${\bf c}\in \OO_L^d$ be the coefficients of the monic minimal polynomial $f_{\bf c}(T)$ of $\zeta_{20p}$ over $L$.
    Note that 
    $$
    d = [L(\zeta_{20p}) : L] \leq [\QQ(\zeta_{20}) : \QQ][L(\zeta_p) : L] \leq \varphi(20)D = 8D
    $$
    by hypothesis.
    Moreover, $A_{\bf c} \cong \OO_L[\zeta_{20p}]\subseteq \Qtr(i)$ and $\Group^{(d)}({\bf c})$ holds because Theorem~\ref{thm:def_muL_in_OOL_if_20th} implies that $M_{\bf c} = \mu(\OO_L[\zeta_{20p}])$.
    Therefore, by Theorem~\ref{thm:identify_R},
    $$
        R_{A_{\bf c}, M_{\bf c}} 
            \cong R_{\OO_L[\zeta_{20p}], \mu(\OO_L[\zeta_{20p}])}
            \subseteq \ZZ + (\zeta_p - 1)\OO_L[\zeta_{20p}]
    $$
    Therefore, by Lemma~\ref{lem:div_by_p}, the discriminant of the monic minimal polynomial of $x$ over $\ZZ$ is divisible by $p$. But it is impossible for this to hold for infinitely many primes $p$, and this contradiction concludes the proof.
\end{proof}

We have now proven Theorem~\ref{thm:main_Zab} because $\Qab$ satisfies the hypotheses of Theorem~\ref{thm:def_Z_in_O_general} with $D = 1$, as it contains all roots of unity.

\section{Following the blueprint of Robinson and Henson for undecidability}
\label{sec:Blueprint_Henson_and_Robinson}
Even when we cannot prove that $\OO_L$ has a first-order definition of $\ZZ$, we can sometimes still prove that $\OO_L$ is undecidable.
To do this, we deploy the material of Section~\ref{sec:def_tor} to fulfill a well-known criterion of C.W. Henson, which was relayed by van den Dries \cite{vandenDries} and builds upon work of J. Robinson \cite{Robinson62}.
Specifically, this criterion gives a simple sufficient condition for when a ring of algebraic integers $A$ admits an interpretation of a weak essentially undecidable arithmetic theory known as R.M. Robinson's $\mathbf Q$.
This gives us the undecidability result that we seek; see \cite{MRUV20}, \cite{Shlapentokh18} and \cite{Springer24} for related use cases, in addition to the work on $\JR$-numbers discussed in Section~\ref{sec:JR}.
The lemma below is a mild generalization compared to the typical formulation of the criterion, as explained in the proof.

\begin{lemma}\label{lem:Robinson-Henson-d_general}
    Let $A$ be a ring of algebraic integers and fix $d\geq 1$. If there exists an $\mathcal L_{\text{ring}}$-formula $\varphi(x_1,\dots, x_d;{\bf b})$ which parameterizes a family of subsets of $A^d$ containing finite sets of arbitrarily large cardinality, then R.M. Robinson's $\mathbf Q$ is interpretable in $A$.
    In particular, the first-order theory of $A$ is undecidable.
\end{lemma}
\begin{proof}
    As mentioned above, the case $d = 1$ is well-known; see \cite{Robinson62} and \cite[\S3.3]{vandenDries}. 
    The general case follows from a simple observation: Let $\pi_i : A^d \to A$ be projection onto the $i$-th coordinate. Clearly, if $S\subseteq A^d$ is finite, then so are all of the projections $\pi_i(S)$. 
    Moreover, $$
        |S| \leq \prod_{i = 1}^d |\pi_i(S)|.
    $$

    Therefore, because $\varphi$ defines a family of subsets of $A^d$ containing finite sets of arbitrarily large size, there is some $i$ so that the projection of these subsets onto the $i$-th coordinate also contains finite subsets of arbitrarily large size. In particular, because the projection is given by binding all variables except $x_i$ to an existential quantifier, this reduces to the $d = 1$ case of Henson and Robinson.
\end{proof}

The fact that Lemma~\ref{lem:Robinson-Henson-d_general} works for general $d$ is useful because Lemma~\ref{lem:interpretations} naturally provides us with formulas on a Cartesian power $\OO_L^d$, as in the following theorem.
 
\begin{theorem}\label{thm:undec}
    Let $L\subseteq \Qtr(i)$.
    If there exists a $D \geq 1$ so that $[L(\zeta) :L]\leq D$ for infinitely many roots of unity $\zeta$, then the first-order theory of $\OO_L$ is undecidable.
\end{theorem}
\begin{proof}
    Let $\zeta_n$ be a primitive $n$-th root of unity for all $n$.
    By hypothesis, there are infinitely many integers $n$ so that
    \[
        [L(\zeta_{20n}) : L] = [L(\zeta_{20n}) : L(\zeta_n)][L(\zeta_{n}) : L]\leq  20D.
    \]
    We deduce that there is some $1\leq d\leq 20D$ and an infinite set $S$ of positive integers so that $[L(\zeta_{20n}) : L]$ is precisely $d$ for all $n\in S$.

    Translate the formula \eqref{eq:psi_def} by the interpretation trick of Lemma~\ref{lem:interpretations} to get a formula over $R^d \cong R[X]/(f_{\bf c}(X))$ for every ring $R$, with additional parameters ${\bf c}$ for the coefficients of the monic polynomial $f_{\bf c}$. Call this formula $\psi^{[d]}({\bf z};{\bf w}, {\bf c})$.

    Let $n\in S$.
    Choose ${\bf c}$ to be the coefficients of the degree-$d$ minimal polynomial of $\zeta_{20n}$ over~$L$.
    These coefficients lie in $\OO_L$ because $\OO_L$ is integrally closed and $\zeta_{20n}$ is an algebraic integer.
    For each prime power $p^e$ dividing $n$, let ${\bf w}$ be the element of $\OO_L^d \cong \OO_L[X]/(f_{\bf c}(X)) \cong \OO_L[\zeta_{20n}]$ corresponding to the primitive $p^e$-th root of unity $\zeta_{20n}^{20n/p^e}$.
    We remark that $\OO_L[\zeta_{20n}]$ also contains a primitive $20$-th root of unity $\zeta_{20n}^{n}$ by construction.
    Therefore, $\psi^{[d]}({\bf z};{\bf w}, {\bf c})$ defines the set of primitive $p^e$-th roots of unity in $\OO_L[\zeta_{20n}]$, interpreted in $\OO_L^d$, by Theorem~\ref{thm:prime_power_if_20th}. 
    This set has cardinality $\varphi(p^e) = p^{e-1}(p-1)$.
    
    Therefore, the family of definable subsets of $\OO_L^d$ parameterized by $\psi^{[d]}$ contains a set of cardinality $p^{e-1}(p-1)$ for every $n\in S$ and every prime power $p^e$ dividing $n$. 
    Because $S$ is infinite, we have $\sup_{n\in S}\max_{p^e \mid n} p^{e-1}(p-1) = \infty$, where the inner maximum is taken over prime powers dividing $n$.
    In conclusion, we have a parameterized family of definable subsets containing finite sets of arbitrarily large cardinality, so we have established undecidability via Lemma~\ref{lem:Robinson-Henson-d_general}.
\end{proof}

\section{JR numbers}\label{sec:JR}

To finish, we make some remarks on the implications of Theorem~\ref{thm:undec} and related constructions for the case of totally imaginary quadratic extensions of totally real fields with a fixed $\JR$-number.

\subsection{Review of \texorpdfstring{$\JR$}{JR}-numbers}

Let $K$ be a totally real field. Recall the following.

\begin{definition}
    Given a totally real number $\alpha$ and $t\in \RR$, write $0\ll \alpha \ll t$ if every conjugate of $\alpha$ lies in the interval $(0,t)$.
    For a totally real subset $X$, write $X_t = \{\alpha\in X : 0\ll \alpha \ll t\}$. The $\JR$-number of $X$ is:
    $$
        \JR(X) := \inf \{t\in \RR : \#X_t = \infty\}.
    $$
    If $\JR(X)$ is finite, then we say that the $\JR$-number is \emph{attained} if the infimum is actually a minimum.
\end{definition}

The usefulness of this definition relies on the following result, which is a straightforward application of Siegel's four-squares theorem \cite{Siegel21}; see \cite[Thm.~4.3]{Springer24} for recent exposition.

\begin{theorem}
    \label{thm:Siegel_formula}
    There is a single existential formula $\phi(x;a,b)$ such that if $K$ is any totally real field and $a,b$ are integers with $b\neq 0$, then $\phi(x;a,b)$ holds over $\OO_K$ if and only if $0\ll x\ll \frac{a}{b}$.
\end{theorem}

In particular, it was an observation of J. Robinson that if $\JR(\OO_K)$ is either infinite or an attained finite value, then $\OO_K$ is undecidable \cite{Robinson62} by an application of J. Robinson's version of Lemma~\ref{lem:Robinson-Henson-d_general} with the parameterized family of subsets given by $\phi(x;a,b)$.
It is also known that $\JR(\OO_K)\in [4,\infty]$.
For more on $\JR$-numbers, see \cite{Castillo-thesis, CVV20, GR19, JV08, MS24, Springer20, Springer24, VV15, VV16, VV26} as a start.

\subsection{Totally imaginary extensions of totally real fields with fixed \texorpdfstring{$\JR$}{JR}-number}

First, observe that Theorem~\ref{thm:undec} applies to rings of integers in totally real fields whose JR-number is at the low extreme.
\begin{corollary}\label{cor:undec_for_JR_min}
    Let $K$ be a totally real field. Assume that $\JR(\OO_K) = 4$ is an attained minimum, or equivalently, assume that $K$ contains $\zeta + \zeta^{-1}$ for infinitely many roots of unity~$\zeta$. Then $\OO_L$ has undecidable first-order theory for all totally imaginary quadratic extensions $L/K$.
\end{corollary}
\begin{proof}
    Lemma~\ref{lem:traces_in_0_4} establishes the fact that $K$ contains $\zeta + \zeta^{-1}$ for infinitely many roots of unity $\zeta$ if and only if $\JR(\OO_K) = 4$ is an attained minimum.
    The proof concludes by observing that $[\QQ(\zeta) : \QQ(\zeta + \zeta^{-1})] \leq 2$ for every root of unity $\zeta$, so the hypotheses of Theorem~\ref{thm:undec} are satisfied with $D = 2$.
\end{proof}

Finally, the analogue of Corollary~\ref{cor:undec_for_JR_min} is also true when $\JR(\OO_K) = \infty$.
The proof of this fact does not require the subgroup of torsion units, but is rather a combination of the main results of \cite{MRS24} with the proof strategy seen in \cite{MRUV20, Springer20, Springer24}.
In particular, we note that this theorem is strictly stronger than the main theorems of \cite{MRUV20, Springer20}.

\begin{theorem}\label{thm:undec_JR_infty}
    Suppose that $K$ is a totally real field with $\JR(\OO_K) = \infty$. If $L/K$ is any totally imaginary quadratic extension, then $\OO_L$ has undecidable first-order theory.
\end{theorem}
\begin{proof}
    If $K = \QQ$, then the result is classical; see \cite{Robinson59}.
    If $K\neq \QQ$, then $L$ is not a quadratic imaginary field, so the subring $R := R_{\OO_L, \OO_L^\times}$ satisfies the containments $\ZZ \subseteq R\subseteq \OO_K$  by \cite[Prop.~2.11]{MRS24}. Recall that this ring is definable in $\OO_L$; see \cite[Cor.~2.7]{MRS24}.

    Now the proof method of \cite[Thm.~4.4]{Springer24} applies immediately.
    Indeed, apply the formula $\phi(x;a,b)$ of Theorem~\ref{thm:Siegel_formula} by relativizing all variables to $R$. For each $t = \frac{a}{b}\in \QQ$, we find
    \begin{align*}
        \{x\in \ZZ : 0\ll x\ll \frac{a}{b}\} &= \{x\in \ZZ : \ZZ\vDash \phi(x;a,b)\}\\
        &\subseteq \{x\in R : R\vDash\phi(x;a,b)\}\\
        &\subseteq  \{x\in \OO_K : \OO_K\vDash \phi(x;a,b)\}
        =\{x\in \OO_K : 0\ll x\ll \frac{a}{b}\}
    \end{align*}
    The inclusions are clear because the inclusion maps $\ZZ\hookrightarrow R$ and $R\hookrightarrow \OO_K$ are injective homomorphisms and existential formulas are preserved under injective homomorphisms.
    Then, as $a$ and $b$ vary, the right-hand side is always finite because $\JR(\OO_K) = \infty$, but the sets on the left-hand side get arbitrarily large. Thus, we are done by Lemma~\ref{lem:Robinson-Henson-d_general}.
\end{proof}

 \bibliographystyle{alpha}
 \bibliography{references}

@phdthesis{Castillo-thesis,
	author = {Marianela {Castillo Fern\'andez}},
	note = {URL: http://repositorio.udec.cl/handle/11594/3003},
	school = {Universidad de Concepci\'on},
	title = {On the Julia Robinson number of rings of totally real algebraic integers in some towers of Nested Square Roots},
	year = 2018}

@article {CVV20,
    AUTHOR = {Castillo, Marianela and Vidaux, Xavier and Videla, Carlos R.},
     TITLE = {Julia {R}obinson numbers and arithmetical dynamic of quadratic
              polynomials},
   JOURNAL = {Indiana Univ. Math. J.},
  FJOURNAL = {Indiana University Mathematics Journal},
    VOLUME = {69},
      YEAR = {2020},
    NUMBER = {3},
     PAGES = {873--885},
      ISSN = {0022-2518,1943-5258},
   MRCLASS = {11R04 (11R09 11R11 11R32 11R80 37P05)},
  MRNUMBER = {4095177},
MRREVIEWER = {Alexandra\ Shlapentokh},
       DOI = {10.1512/iumj.2020.69.7928},
       URL = {https://doi.org/10.1512/iumj.2020.69.7928},
}

@article{DKMWY25,
    AUTHOR = {Daans, Nicolas and Kala, V{\'i}t{\v e}zslav and Man, Siu Hang and
              Widmer, Martin and Yatsyna, Pavlo},
     TITLE = {Most totally real fields do not have universal forms or the
              {N}orthcott property},
   JOURNAL = {Proc. Natl. Acad. Sci. USA},
  FJOURNAL = {Proceedings of the National Academy of Sciences of the United
              States of America},
    VOLUME = {122},
      YEAR = {2025},
    NUMBER = {20},
     PAGES = { \ Paper No. e2419414122, 8},
      ISSN = {0027-8424,1091-6490},
   MRCLASS = {11E12 (11R80)},
  MRNUMBER = {4964796},
       DOI = {10.1073/pnas.2419414122},
       URL = {https://doi.org/10.1073/pnas.2419414122},
}

@article{DF21,
	 AUTHOR = {Dittmann, Philip and Fehm, Arno},
     TITLE = {Nondefinability of rings of integers in most algebraic fields},
   JOURNAL = {Notre Dame J. Form. Log.},
  FJOURNAL = {Notre Dame Journal of Formal Logic},
    VOLUME = {62},
      YEAR = {2021},
    NUMBER = {3},
     PAGES = {589--592},
      ISSN = {0029-4527,1939-0726},
   MRCLASS = {11U09 (03C40 12L12)},
  MRNUMBER = {4323047},
MRREVIEWER = {Ricardo\ Bianconi},
       DOI = {10.1215/00294527-2021-0029},
       URL = {https://doi.org/10.1215/00294527-2021-0029},
}

@article{vandenDries,
 AUTHOR = {van den Dries, Lou},
     TITLE = {Elimination theory for the ring of algebraic integers},
   JOURNAL = {J. Reine Angew. Math.},
  FJOURNAL = {Journal f\"ur die Reine und Angewandte Mathematik. [Crelle's
              Journal]},
    VOLUME = {388},
      YEAR = {1988},
     PAGES = {189--205},
      ISSN = {0075-4102,1435-5345},
   MRCLASS = {03C60 (11D72 12L05)},
  MRNUMBER = {944190},
MRREVIEWER = {\c S.\ A.\ Basarab},
       DOI = {10.1515/crll.1988.388.189},
       URL = {https://doi.org/10.1515/crll.1988.388.189},
}

@incollection{FHV94,
	author = {Fried, Michael D. and Haran, Dan and V\"{o}lklein, Helmut},
	booktitle = {Arithmetic geometry ({T}empe, {AZ}, 1993)},
	doi = {10.1090/conm/174/01849},
	mrclass = {12D15 (11G25 12E25 12F12)},
	mrnumber = {1299732},
	mrreviewer = {Mieczys\l aw Kula},
	pages = {1--34},
	publisher = {Amer. Math. Soc., Providence, RI},
	series = {Contemp. Math.},
	title = {Real {H}ilbertianity and the field of totally real numbers},
	url = {https://doi.org/10.1090/conm/174/01849},
	volume = {174},
	year = {1994}}

@article {GR19,
    AUTHOR = {Gillibert, Pierre and Ranieri, Gabriele},
     TITLE = {Julia {R}obinson numbers},
   JOURNAL = {Int. J. Number Theory},
  FJOURNAL = {International Journal of Number Theory},
    VOLUME = {15},
      YEAR = {2019},
    NUMBER = {8},
     PAGES = {1565--1599},
      ISSN = {1793-0421,1793-7310},
   MRCLASS = {11R04 (11R80 11U05)},
  MRNUMBER = {3994148},
MRREVIEWER = {Alexandra\ Shlapentokh},
       DOI = {10.1142/S1793042119500908},
       URL = {https://doi.org/10.1142/S1793042119500908},
}

@book {Hodges,
    AUTHOR = {Hodges, Wilfrid},
     TITLE = {Model theory},
    SERIES = {Encyclopedia of Mathematics and its Applications},
    VOLUME = {42},
 PUBLISHER = {Cambridge University Press, Cambridge},
      YEAR = {1993},
     PAGES = {xiv+772},
      ISBN = {0-521-30442-3},
   MRCLASS = {03-01 (03-02 03Cxx)},
  MRNUMBER = {1221741},
MRREVIEWER = {J.\ M.\ Plotkin},
       DOI = {10.1017/CBO9780511551574},
       URL = {https://doi.org/10.1017/CBO9780511551574},
}

@article {JV08,
    AUTHOR = {Jarden, Moshe and Videla, Carlos R.},
     TITLE = {Undecidability of families of rings of totally real integers},
   JOURNAL = {Int. J. Number Theory},
  FJOURNAL = {International Journal of Number Theory},
    VOLUME = {4},
      YEAR = {2008},
    NUMBER = {5},
     PAGES = {835--850},
      ISSN = {1793-0421,1793-7310},
   MRCLASS = {11U05 (03B25 03C07 12E30)},
  MRNUMBER = {2458847},
MRREVIEWER = {Alexandra\ Shlapentokh},
       DOI = {10.1142/S1793042108001705},
       URL = {https://doi.org/10.1142/S1793042108001705},
}

@misc{Kartas21,
      title={Diophantine problems over $\mathbb{Z}^{ab}$ modulo prime numbers}, 
      author={Konstantinos Kartas },
      year={2021},
      eprint={2104.06741},
      archivePrefix={arXiv},
      primaryClass={math.NT},
      url={https://arxiv.org/abs/2104.06741}, 
      howpublished={arXiv preprint: \url{https://arxiv.org/abs/2104.06741}}, 
}

@incollection {Koenigsmann14,
    AUTHOR = {Koenigsmann, Jochen},
     TITLE = {Undecidability in number theory},
 BOOKTITLE = {Model theory in algebra, analysis and arithmetic},
    SERIES = {Lecture Notes in Math.},
    VOLUME = {2111},
     PAGES = {159--195},
 PUBLISHER = {Springer, Heidelberg},
      YEAR = {2014},
      ISBN = {978-3-642-54935-9; 978-3-642-54936-6},
   MRCLASS = {11U05 (03B25 03D35)},
  MRNUMBER = {3330199},
MRREVIEWER = {Kirsten\ Eisentr\"ager},
       DOI = {10.1007/978-3-642-54936-6\_5},
       URL = {https://doi.org/10.1007/978-3-642-54936-6_5},
}

@article {Kronecker1857,
    AUTHOR = {Kronecker, L.},
     TITLE = {Zwei {S}\"atze \"uber {G}leichungen mit ganzzahligen
              {C}oefficienten},
   JOURNAL = {J. Reine Angew. Math.},
  FJOURNAL = {Journal f\"ur die Reine und Angewandte Mathematik. [Crelle's
              Journal]},
    VOLUME = {53},
      YEAR = {1857},
     PAGES = {173--175},
      ISSN = {0075-4102,1435-5345},
   MRCLASS = {99-04},
  MRNUMBER = {1578994},
       DOI = {10.1515/crll.1857.53.173},
       URL = {https://doi.org/10.1515/crll.1857.53.173},
}

@article{MRUV20,
	author = {Mart\'{\i}nez-Ranero, Carlos and Utreras, Javier and Videla, Carlos R.},
	doi = {10.1090/proc/14849},
	fjournal = {Proceedings of the American Mathematical Society},
	issn = {0002-9939},
	journal = {Proc. Amer. Math. Soc.},
	mrclass = {11U05 (03B25 11R11)},
	mrnumber = {4055926},
	number = {3},
	pages = {961--964},
	title = {Undecidability of {$\mathbb Q^{(2)}$}},
	url = {https://doi.org/10.1090/proc/14849},
	volume = {148},
	year = {2020}}

@article {MR20,
    AUTHOR = {Mazur, Barry and Rubin, Karl},
     TITLE = {Big fields that are not large},
   JOURNAL = {Proc. Amer. Math. Soc. Ser. B},
  FJOURNAL = {Proceedings of the American Mathematical Society. Series B},
    VOLUME = {7},
      YEAR = {2020},
     PAGES = {159--169},
      ISSN = {2330-1511},
   MRCLASS = {11G05 (11R04 11U05 14G05)},
  MRNUMBER = {4173816},
MRREVIEWER = {Kirsten\ Eisentr\"ager},
       DOI = {10.1090/bproc/57},
       URL = {https://doi.org/10.1090/bproc/57},
}

@article {MRS24,
    AUTHOR = {Mazur, Barry and Rubin, Karl and Shlapentokh, Alexandra},
     TITLE = {Defining {$\mathbb Z$} using unit groups},
   JOURNAL = {Acta Arith.},
  FJOURNAL = {Acta Arithmetica},
    VOLUME = {214},
      YEAR = {2024},
     PAGES = {235--255},
      ISSN = {0065-1036,1730-6264},
   MRCLASS = {11U05},
  MRNUMBER = {4772285},
MRREVIEWER = {Fred\ W.\ Roush},
       DOI = {10.4064/aa230505-6-6},
       URL = {https://doi.org/10.4064/aa230505-6-6},
}

@article {MS24,
    AUTHOR = {Mu\~noz Sandoval, Carlos},
     TITLE = {New values of the {J}ulia {R}obinson number},
   JOURNAL = {Cubo},
  FJOURNAL = {Cubo. A Mathematical Journal},
    VOLUME = {26},
      YEAR = {2024},
    NUMBER = {3},
     PAGES = {387--406},
      ISSN = {0716-7776,0719-0646},
   MRCLASS = {11U05 (03B25 11R80)},
  MRNUMBER = {4843276},
       DOI = {10.56754/0719-0646.2603.387},
       URL = {https://doi.org/10.56754/0719-0646.2603.387},
}

@article {Robinson59,
    AUTHOR = {Robinson, Julia},
     TITLE = {The undecidability of algebraic rings and fields},
   JOURNAL = {Proc. Amer. Math. Soc.},
  FJOURNAL = {Proceedings of the American Mathematical Society},
    VOLUME = {10},
      YEAR = {1959},
     PAGES = {950--957},
      ISSN = {0002-9939,1088-6826},
   MRCLASS = {02.00},
  MRNUMBER = {112842},
MRREVIEWER = {R.\ M.\ Martin},
       DOI = {10.2307/2033628},
       URL = {https://doi.org/10.2307/2033628},
}

@incollection{Robinson62,
	author = {Robinson, Julia},
	booktitle = {Studies in mathematical analysis and related topics},
	mrclass = {02.74 (10.80)},
	mrnumber = {0146083},
	mrreviewer = {S. Ja\'{s}kowski},
	pages = {297--304},
	publisher = {Stanford Univ. Press, Stanford, Calif},
	title = {On the decision problem for algebraic rings},
	year = {1962}}

@article{Shlapentokh18,
    AUTHOR = {Shlapentokh, Alexandra},
     TITLE = {First-order decidability and definability of integers in
              infinite algebraic extensions of the rational numbers},
   JOURNAL = {Israel J. Math.},
  FJOURNAL = {Israel Journal of Mathematics},
    VOLUME = {226},
      YEAR = {2018},
    NUMBER = {2},
     PAGES = {579--633},
      ISSN = {0021-2172,1565-8511},
   MRCLASS = {03C40 (03B25 11U05)},
  MRNUMBER = {3819703},
MRREVIEWER = {Xavier\ Vidaux},
       DOI = {10.1007/s11856-018-1708-y},
       URL = {https://doi.org/10.1007/s11856-018-1708-y},
}

@article{Siegel21,
	author = {Siegel, Carl},
	doi = {10.1007/BF01203627},
	fjournal = {Mathematische Zeitschrift},
	issn = {0025-5874},
	journal = {Math. Z.},
	mrclass = {DML},
	mrnumber = {1544496},
	number = {3-4},
	pages = {246--275},
	title = {Darstellung total positiver {Z}ahlen durch {Q}uadrate},
	url = {https://doi.org/10.1007/BF01203627},
	volume = {11},
	year = {1921}}

@article{Springer20,
	author = {Springer, Caleb},
	doi = {10.1090/proc/15153},
	fjournal = {Proceedings of the American Mathematical Society},
	issn = {0002-9939},
	journal = {Proc. Amer. Math. Soc.},
	mrclass = {11U05 (03D35 11R04 11R27 11R80)},
	mrnumber = {4143388},
	mrreviewer = {Alexandra Shlapentokh},
	number = {11},
	pages = {4705--4715},
	title = {Undecidability, unit groups, and some totally imaginary infinite extensions of {$\mathbb{Q}$}},
	url = {https://doi.org/10.1090/proc/15153},
	volume = {148},
	year = {2020}}

@article {Springer24,
    AUTHOR = {Springer, Caleb},
     TITLE = {Definability and decidability for rings of integers in totally
              imaginary fields},
   JOURNAL = {Bull. Lond. Math. Soc.},
  FJOURNAL = {Bulletin of the London Mathematical Society},
    VOLUME = {56},
      YEAR = {2024},
    NUMBER = {1},
     PAGES = {306--318},
      ISSN = {0024-6093,1469-2120},
   MRCLASS = {11U05 (11R04 12L05)},
  MRNUMBER = {4710197},
MRREVIEWER = {Alexandra\ Shlapentokh},
       DOI = {10.1112/blms.12933},
       URL = {https://doi.org/10.1112/blms.12933},
}

@article {VV15,
    AUTHOR = {Vidaux, Xavier and Videla, Carlos R.},
     TITLE = {Definability of the natural numbers in totally real towers of
              nested square roots},
   JOURNAL = {Proc. Amer. Math. Soc.},
  FJOURNAL = {Proceedings of the American Mathematical Society},
    VOLUME = {143},
      YEAR = {2015},
    NUMBER = {10},
     PAGES = {4463--4477},
      ISSN = {0002-9939,1088-6826},
   MRCLASS = {03B25 (03C40 11R80 11U05)},
  MRNUMBER = {3373945},
MRREVIEWER = {Ricardo\ Bianconi},
       DOI = {10.1090/S0002-9939-2015-12592-0},
       URL = {https://doi.org/10.1090/S0002-9939-2015-12592-0},
}

@article {VV16,
    AUTHOR = {Vidaux, Xavier and Videla, Carlos R.},
     TITLE = {A note on the {N}orthcott property and undecidability},
   JOURNAL = {Bull. Lond. Math. Soc.},
  FJOURNAL = {Bulletin of the London Mathematical Society},
    VOLUME = {48},
      YEAR = {2016},
    NUMBER = {1},
     PAGES = {58--62},
      ISSN = {0024-6093,1469-2120},
   MRCLASS = {11U05 (03D35 11R80)},
  MRNUMBER = {3455748},
MRREVIEWER = {Alexandra\ Shlapentokh},
       DOI = {10.1112/blms/bdv089},
       URL = {https://doi.org/10.1112/blms/bdv089},
}

@article {VV26,
    AUTHOR = {Vidaux, Xavier and Videla, Carlos R.},
     TITLE = {An approach to {J}ulia {R}obinson numbers through the lattice
              of subfields},
   JOURNAL = {J. Pure Appl. Algebra},
  FJOURNAL = {Journal of Pure and Applied Algebra},
    VOLUME = {230},
      YEAR = {2026},
    NUMBER = {7},
     PAGES = {Paper No. 108295, 19},
      ISSN = {0022-4049,1873-1376},
   MRCLASS = {12F05 (12F10)},
  MRNUMBER = {5072925},
       DOI = {10.1016/j.jpaa.2026.108295},
       URL = {https://doi.org/10.1016/j.jpaa.2026.108295},
}

@book {Washington,
    AUTHOR = {Washington, Lawrence C.},
     TITLE = {Introduction to cyclotomic fields},
    VOLUME = {83},
   EDITION = {Second},
 PUBLISHER = {Springer-Verlag, New York},
      YEAR = {1997},
     PAGES = {xiv+487},
      ISBN = {0-387-94762-0},
   MRCLASS = {11R18 (11-01 11-02 11R23)},
  MRNUMBER = {1421575},
MRREVIEWER = {T.\ Mets\"ankyl\"a},
       DOI = {10.1007/978-1-4612-1934-7},
       URL = {https://doi.org/10.1007/978-1-4612-1934-7},
}

\end{document}